\documentclass[12pt]{amsart}
\usepackage{amsfonts}
\usepackage{amsthm}
\usepackage{amsmath}
\usepackage{amscd}
\usepackage[latin2]{inputenc}
\usepackage{t1enc}
\usepackage[mathscr]{eucal}
\usepackage{indentfirst}
\usepackage{graphicx}
\usepackage{graphics}
\usepackage{pict2e}
\usepackage{epic}
\usepackage{mathrsfs}
\numberwithin{equation}{section}
\usepackage[margin=2.9cm]{geometry}
\usepackage{epstopdf} 
\usepackage{mathabx}
\usepackage{tikz-cd}
\usepackage{hyperref}
\usepackage{array}

\theoremstyle{plain}
\newtheorem{thm}{Theorem}[section]
\newtheorem{lem}[thm]{Lemma}
\newtheorem{prop}[thm]{Proposition}
\newtheorem{cor}[thm]{Corollary}

\theoremstyle{definition}

\theoremstyle{remark}

\newcommand{\im}{\operatorname{im}}

\newcommand{\BB}[1]{\mathbb{#1}} 
\newcommand{\curly}[1]{\mathscr{#1}} 

\newcommand{\QQ}{\BB{Q}}
\newcommand{\ZZ}{\BB{Z}}
\newcommand{\PP}{\BB{P}}
\newcommand{\GG}{\BB{G}}

\newcommand{\Aut}{\operatorname{Aut}}

\newcommand{\Pic}{\operatorname{Pic}}
\newcommand{\Gal}{\textup{Gal}}

\newcommand{\et}{\text{\'{e}t}}
\newcommand{\Br}{\operatorname{Br}}
\newcommand{\spec}{\operatorname{Spec}}

\newcommand{\Ho}{\textup{H}}
\begin{document}

\title{Vanishing of the Brauer Group of a del Pezzo surface of degree 4}
\author{Manar Riman}
\maketitle

\begin{abstract}
We explicitly construct a del Pezzo surface $X$ of degree 4 over a field $k$ such that $\Ho^1(k,\Pic\overline X)\simeq\ZZ/2\ZZ$ while $\Br X/\Br k$ is trivial. This proves that the algorithm to compute the Brauer group in \cite{BT} cannot be generalized in some cases.

\end{abstract}

\section{Introduction}
The Brauer group of a field $k$, denoted by $\Br k$, was defined by Richard Brauer in 1929. It classifies central simple algebras over a field $k$ with respect to the Morita equivalence. Two central simple algebras $\curly A$ and $\curly B$ are said to be Morita equivalent over $k$ if there there exist two positive integers $n$ and $m$ such that $\curly A\otimes \textup{M}_n(k)\simeq\curly B\otimes \textup{M}_m(k)$ as $k$-algebras. Moreover, $\Br k$ is isomorphic to the Galois cohomology group $\Ho^2(k,k_s^{\times})$; we say this is the cohomological definition of the Brauer group. The cohomological definition of the Brauer group of a field can be generalized to define the Brauer group of a scheme $X$ denoted by $\Br X$ as the \'etale cohomology $\Ho_{\et}^2(X,\GG_m)$, where $\GG_m$ is the sheaf of units on $X$. When $X$ is projective, geometrically integral and smooth, $\Br X$ is realized concretely as the equivalence classes of Azumaya algebras over $X$. This has opened the door to many arithmetic and geometric applications of the Brauer group.

Various theorems and algorithms have been implemented to compute the quotient of the Brauer group $\Br X/\im(\Br k\rightarrow\Br X)$ for certain schemes. For example, such algorithms have been implemented for most cubic surfaces; See \cite{CTcubic}, \cite{EJ} and \cite{EJ12}. Algorithms have also been implemented for some del Pezzo surfaces of degree 2 as in \cite{Corn}, and for most del Pezzo surfaces of degree 4 as in \cite{BBFL}, and \cite{BT}. 

In this paper we highlight a part of the algorithm to compute $\Br X/\im(\Br k\rightarrow\Br X)$ for a del Pezzo surface of degree 4 as in \cite[Section 4.1]{BT}. Furthermore, we prove that it cannot be generalized in some cases. 

Let $X$ be a del Pezzo surface of degree 4. By embedding $X$ anticanonically into $\PP^4$, we view it as the intersection of two quadrics as in \cite[Proposition 3.26]{Wit}. The two quadrics $Q$ and $Q^{'}$ define a pencil $\{\lambda Q+\mu Q^{'}:[\lambda,\mu]\in\PP^1\}$. By \cite[Proposition 3.26]{Wit}, the pencil has five degenerate geometric fibers which are rank 4 quadrics. Let $\curly S$ be the degree 5 subscheme of $\PP^1$ representing the degeneracy locus of the pencil. For every closed point $T\in\curly S$, denote by $k(T)$ the residue field of $T$ and by $Q_T$ the corresponding quadric in the pencil. 

The proof of the algorithm to compute $\Br X/\im(\Br k\rightarrow\Br X)$ for a del Pezzo surface of degree 4 as in \cite[Section 4.1]{BT} uses the exact sequence, resulting from the Hochschild-Serre spectral sequence,
$$
0\rightarrow\Pic X\rightarrow (\Pic \overline X)^{G_k}\rightarrow\Br k\rightarrow\Br_1 X\rightarrow \Ho^1(k,\Pic\overline X)\rightarrow \Ho^3(k,\GG_m),
$$
where $G_k$ is the absolute Galois group of $k$ and $\Br_1X=\ker(\Br X\rightarrow\Br\overline X)$. The Hochschild-Serre spectral sequence yields the isomorphism
$$
\frac{\Br_1 X}{\im(\Br k\rightarrow\Br X)}\simeq \Ho^1(k,\Pic\overline X)
$$
under some arithmetic assumptions related to the solvability of the quadrics $Q_T$ over $k(T)$ for $T\in\curly S$. Depending on the field $k$, the map $\Ho^1(k,\Pic\overline X)\rightarrow \Ho^3(k,\GG_m)$ might be non trivial as we prove for a certain del Pezzo surfaces of degree 4 in this paper. In other words, this proves that the arithmetic assumption in the algorithm in \cite[Section 4.1]{BT} is necessary; changing the arithmetic input of the algorithm leads to a different outcome. In particular, we prove the following theorem.
\begin{thm}
\label{thm:main}
Let $k=\QQ^{cycl}(a,b,c)$ where $a,b$ and $c$ are independent transcendental elements. Let $X$ be the del Pezzo surface of degree 4 in $\PP^4_k$ defined by the intersection of the following two quadrics
\begin{equation*}
\begin{split}
Q&:ax_0^2+bx_1^2+x_2^2+cx_4^2=0\\
Q^{'}&:bcx_0^2+x_1^2+x_2^2+ax_3^2=0.
\end{split}
\end{equation*}
Then $\Ho^1(k,\Pic \overline X)\simeq\ZZ/2\ZZ$ while $\frac{\Br X}{\Br k}$ is trivial.
\end{thm}

Uematsu has done similar work for an affine diagonal quadric \cite{quadric} and for a diagonal cubic surface \cite{cubic}. In contrast to his proofs, our proof does not rely on computing the boundary map $d_2^{1,1}: H^1(k,\Pic\overline X)\rightarrow H^3(k,\GG_m)$ of the Hochschild-Serre spectral sequence. Instead it relies on the algorithm in \cite[Section 4.1]{BT} and the work done by Harpaz \cite{Harpaz}. In \cite{Harpaz}, Harpaz proves that for a specific example of an affine diagonal quadric $U$ over a field $F$, $\Br U/\Br F$ vanishes while $H^1(F,\Pic\overline U)\simeq \ZZ/2\ZZ$. In his argument, he base extends $X$ to a field $L$; over this field he shows that $\Br X_L/\Br L\simeq \ZZ/2\ZZ$. Finally, he proves that the generating class $\curly A$ of $\Br X_L/\Br L$ is not in the image of $\Br X/\Br k$. We carry a similar proof to Harpaz's argument; we base extend to the field $L=k(\sqrt a)$. For the explicit computation of the class of the algebra $\curly A$ that generates $\Br X_L/\Br L$, we use the algorithm in \cite[Section 4.1]{BT}. 

\textbf{Outline:}
Section 2 provides some background about del Pezzo surfaces of degree 4. In particular, we highlight a part of the algorithm in \cite[Section 4.1]{BT} to relate $H^1(k,\Pic\overline X)$ to $\Br X/\Br k$. We prove the main theorem in Section 3.

\textbf{Notation:}
Throughout this paper, we use $k$ to denote a field and $X$ to denote a scheme. We denote by $\overline k$ a fixed algebraic closure of $k$. We denote by $\overline X$ the base change of $X$ to $\overline k$. Let $G_k$ denote the absolute Galois group of $k$. Denote by $H^i(k,A)$ the $i$-th group cohomology of the Galois group $G_k$ and the $G_k$-module $A$.

\section*{Acknowlegements}

I thank my advisor, Bianca Viray, for suggesting the problem and for the many helpful discussions. I also thank Yonatan Harpaz for emailing us his typed notes \cite{Harpaz}.

\section{Background}
\subsection{ Background on del Pezzo Surfaces of Degree 4}

We summarize some facts that we need in this paper about del Pezzo surfaces of degree 4 and we fix some notation. For further details, this information can be found in \cite[Section 2]{BT}, \cite[Section 1]{tony}, and \cite{Wit}.

Let $Y$ be a del Pezzo surface of degree 4 over a field $K$. By embedding it anticanonically into $\PP^4$, we view it as the intersection of two quadrics $Q$ and $Q^{'}$ as in \cite[Proposition 3.26]{Wit}. The two quadrics $Q$ and $Q^{'}$ define a pencil $\{\lambda Q+\mu Q^{'}:[\lambda,\mu]\in\PP^1\}$. By \cite[Proposition 3.26]{Wit}, the pencil has five degenerate geometric fibers which are rank 4 quadrics.

Let $\curly S$ be the degree 5 subscheme of $\PP^1$ representing the degeneracy locus of the pencil. For every closed point $T\in\curly S$, denote by $K(T)$ the residue field of $T$ and by $Q_T$ the corresponding quadric in the pencil. Let $\epsilon_{T}$ be the discriminant of a smooth rank 4 quadric obtained by restricting $Q_T$ to a hyperplane $H_T$ in $\PP^4$ not containing the vertex of $Q_T$. By \cite[Section 3.4.1]{Wit}, the square class of $\epsilon_T$ does not depend on the choice of $H_T$. So we consider $\epsilon_T$ as an element in $K(T)/K(T)^{\times 2}$. Over $\overline K$, let $\curly S(\overline K)=\{t_0,\ldots ,t_4\}$ where $t_0,\ldots, t_4\in \PP^1(\overline k)$. Denote by $K(t_i)$ the smallest field contained in $\overline k$ and containing $t_i$, and by $Q_{t_i}$ the corresponding quadric. Let $\epsilon_{t_i}$ be the discriminant of the smooth rank 4 quadric obtained by restricting $Q_{t_i}$ to a hyperplane $H_i$ in $\PP^4$ not containing the vertex of $Q_{t_i}$. By \cite[Section 3.4.1]{Wit}, the square class of $\epsilon_{t_i}$ does not depend on the choice of $H_i$. So we consider $\epsilon_{t_i}$ as an element in $K(t_i)/K(t_i)^{\times 2}$

Now we turn our attention to the Picard group of $\overline Y$ and the Galois action on its classes. By \cite[Theorem 1.6]{tony}, $\overline Y$ is the blowup of $\PP^2$ at 5 points. Let $\{e_1,\ldots,e_5\}$ be the classes of the exceptional divisors associated to the blown up points. Let $l$ be the class of the pullback of a line in $\PP^2$ that does not not pass through any of the blown up points. By \cite[Proposition V.3.2]{Hart}, $\Pic \overline Y\simeq \ZZ^6$ and admits $\{e_1,\ldots,e_5,l\}$ as a basis. Furthermore, we will describe $\Pic\overline Y$ in terms of some conic classes that will be useful for the algorithm in \cite[Section 4.1]{BT}.

\begin{thm}{\cite[Theorem 2]{BBFL}}
\label{Theorem:conics}
Let $Y$ be a del Pezzo surface of degree $4$ over a field $K$. There are $10$ families of conics on $Y$. Moreover, their classes $C_0,\ldots,C_4,C_0^{'},\ldots,C_4^{'}$ in $\Pic\overline Y$ can be written in terms of the basis of $\Pic\overline Y$ as $C_i=l-e_i$ and $C_i^{'}=H-C_i$ where $H$ is the hyperplane class of $Y$. Over $\overline K$ the conics in each class form a pencil.\qed
\end{thm}

As in \cite[Section 2.3]{BT}, we define the conic classes $C_0,\ldots, C_4,C_0^{'}, \ldots, C_4^{'}$ in $\Pic\overline Y$ as follows. For $i=0,\ldots,4$, let $K_i$ be a finite extension of $K$ such that the rank 4 quadric $Q_{t_i}$ has a smooth $K_i$-point and such that $[K_i(\sqrt\epsilon_{t_i}):K_i]=[K(t_i)(\sqrt \epsilon_{t_i}):K(t_i)]$. Let $P_i$ be any smooth $K_i$-point on $Q_{t_i}$ and $H_{P_i}$ be the hyperplane tangent to $Q_{t_i}$ at $P_i$. By \cite[Lemma 2.1]{BT}, we have $Q_{t_i}\cap H_{P_i}=L_{P_i}\cup L_{P_i}^{'}$ for some planes $L_{P_i}$ and $L_{P_i}^{'}$ defined over $K_i(\sqrt\epsilon_{t_i})$. We define $C_{P_i}=Y\cap L_{P_i}$ and $C_{P_i}^{'}=Y\cap L_{P_i}^{'}$. We have, 
\[C_{P_i}=W\cap Q_{t_i}\cap L_{P_i}=W\cap L_{P_i} \textup{ and } C_{P_i}^{'}=W\cap Q_{t_i}\cap L_{P_i}^{'}=W\cap L_{P_i}^{'}\]
for some smooth quadric $W$ in the pencil associated to $Y$. Hence $C_{P_i}$ and $C_{P_i}^{'}$ are conics on $Y$. A different choice of $K_i$-smooth point $\tilde P_i$ on $Q_{t_i}$ leads to two different planes $L_{\tilde P_i}$ and $L_{\tilde P_i}^{'}$, and hence two different conics. Because $Q_{t_i}$ is a singular rank 4 quadric over $\overline K$, it is a cone over a smooth quadric $Q^{''}$ in $\PP^3$. The two families of lines on $Q^{''}$ induce two families of planes on $Q_{t_i}$. So $L_{P_i},L_{P_i}^{'},L_{\tilde P_i}$ and $L_{\tilde P_i}^{'}$ belong to two pencils. So without loss of generality we may assume that $C_{P_i}\sim C_{\tilde{P_i}}$ and $C_{P_i}^{'}\sim C_{\tilde{P_i}}^{'}$. We denote the classes of these conics by $C_i$ and $C_i^{'}$ which are independent of the chosen point on $Q_{t_i}$. By Theorem \ref{Theorem:conics}, the classes of these conics are the 10 possible classes of conics on $Y$ and $C_i^{'}=H-C_{i}$ for every $i\in\{0, \ldots, 4\}$ where $H$ is the hyperplane class of $Y$.

\begin{prop}{\cite[Proposition 2.2]{BT}}
\label{prop:pic}
After possibly interchanging the $C_i$ and the $C_i^{'}$ for some indices $i$, we may assume that the Picard group $\Pic\overline Y\cong\ZZ^6$ is freely generated by the following classes
$$
\frac{1}{2}(H+C_0+C_1+C_2+C_3+C_4),C_0,C_1,C_2,C_3,C_4
$$
where $H$ is the hyperplane class of $Y$.\qed
\end{prop}

Consider $\sigma\in G_K$, and let $\sigma ^{'}\colon\spec \overline K\rightarrow \spec\overline K$ be the corresponding morphism on schemes. After base changing to $\overline K$ we get
$$
\textup{id}_Y\times \sigma ^{'}\colon \overline Y=Y\times _{K}\spec\overline K\rightarrow \overline Y.
$$
The morphism above induces an automorphism on $\Pic \overline Y$. This defines the action of $G_K$ on $\Aut (\Pic \overline Y)$. This action fixes the canonical class and the intersection multiplicity \cite[Theorem 23.8]{CubicForms}.

Let $\Gamma$ be the graph of ten vertices indexed by $C_i$ and $C_i^{'}$ whose edges join $C_i$ and $C_i^{'}$.  The group $\Aut(\Gamma)$ is the semi-direct product $(\ZZ/2\ZZ)^5\rtimes S_5$ where the $\ZZ/2\ZZ$ entries represent the exchanges for each $\{C_i,C_i^{'}\}$ and the $S_5$ entry represents a permutation of the sets $\{C_i,C_i^{'}\}$ for $i\in\{0,\ldots,4\}$. Let $\curly O(K_Y^{\bot})$ be the subgroup of $\Aut(\Gamma)$ that fixes the orthogonal complement of $K_Y$ in $\Pic\overline Y$. By the discussion in \cite[p:8-10]{KST}, there is a natural embedding $\curly O(K_Y^{\bot})\hookrightarrow \Aut(\Gamma)$ of index 2. Moreover by \cite[p:8-10]{KST} the image of this embedding constitutes of all automorphisms that are the product of an even number of exchanges and an element of $S_5$.

Fix $T\in\curly S$. Let $\Gamma _T$ be subgraph with $2\deg(T)$ vertices indexed by $C_i$ and $C_i^{'}$ for $t_i\in T(\overline K)$.

\begin{prop}{\cite[Proposition 2.3]{BT}}
\label{prop:action}
The action of $G_K$ on $\Pic\overline Y$ induces an action on $\Gamma$ that factors through
$$
\Pi_{T\in\curly S}\Aut(\Gamma_T)\cap\curly O(K_Y^{\bot})\subset\Aut(\Gamma).
$$
Moreover, $G_K$ acts transitively on $\{C_i,C_i^{'}:t_i\in T(\overline K)\}$ if and only if $\epsilon _T\notin K(T)^{\times 2}$.
\end{prop}

\begin{proof}
By the discussion above, the action of $G_K$ factors through $\Pi_{T\in\curly S}\Aut(\Gamma_T)\cap\curly O(K_Y^{\bot})$. Moreover, by definition the pair $\{C_i,C_i^{'}\}$ is defined over $K(t_i)$ and each individual conic over $K(t_i,\sqrt{\epsilon_{T}})$. So $\epsilon_T\notin K(T)^{\times 2}$ is exactly the condition for a transitive action of $G_K$ on $\{C_i,C_i^{'}\}$.
\end{proof}

\subsection{An algorithm to relate $\Ho^1(K,\Pic\overline Y)$ and $\Br Y/\Br K$}
Throughout this section, we assume that there is a subscheme $\curly T\subset \curly S$ that satisfies the conditions
$$
(*)\hspace{5ex}\deg(\curly T)=2,~~\Pi_{T\in\curly T}N_{K(T)/K}(\epsilon_T)\in K^{\times 2},\textup{and   }\epsilon_T\notin K(T)^{\times 2}\textup{ for every }T\in\curly T.\hspace{5ex}
$$
Because $\curly T$ satisfies (*), Lemma \cite[Lemma 3.1]{BT} allows us to assume that $\epsilon_T\in \im(K^{\times}/K^{\times 2}\rightarrow K(T)^{\times}/K(T)^{\times 2})$. Define $K_{\curly T}:=K(\sqrt{\epsilon_T})$ which is independent of $T$ because $\deg(\curly T)=2$ and $\Pi_{T\in\curly T}N_{K(T)/K}(\epsilon_T)\in K^{\times 2}$. By Lemma \cite[Lemma 3.1]{BT} and since $\epsilon_T\notin K(T)^{\times 2}$, $K_{\curly T}$ is a quadratic extension of $K$.

\begin{prop}
\label{prop:H1}
The cocycle in $\Ho^1(K,\Pic \overline Y)$ given by
$$
\sigma\mapsto\begin{cases}-H+\Sigma _{t_i\in\curly T(\overline K)}C_i & \textup{ if }\sigma\notin G_{K_{\curly T}}\\ 0 & \textup{ otherwise}\end{cases}
$$
is non trivial if and only if there exists $T\in\curly S-\curly T$ such that $\epsilon_{T}\notin K(T)^{\times 2}$.
\end{prop}

\begin{proof}
This proof can be found in \cite[Section 3.2, Proposition 3.3]{BT} assuming that $Q_T$ has a smooth $K(T)$-point for all $T\in\curly T$. However the proof works without the assumption that $Q_T$ has a smooth $K(T)$-point for all $T\in\curly T$. We repeat the proof for the reader's convenience.

The long exact sequence on cohomology associated to the short exact sequence
$$
0\rightarrow\Pic\overline Y\xrightarrow{\times 2}\Pic \overline Y\rightarrow (\Pic\overline Y/2\Pic\overline Y)\rightarrow 0
$$
induces an isomorphism
\begin{equation}
\label{piciso1}
\frac{(\Pic\overline Y/2\Pic\overline Y)^{G_K}}{((\Pic \overline Y)^{G_K}/2(\Pic\overline Y)^{G_K})}\xrightarrow{\sim}\Ho^1(K,\Pic\overline Y)[2];\quad D\mapsto (\sigma\mapsto\frac{1}{2}(d-\sigma d))
\end{equation}
where $d$ is a lift of $D$ to $\Pic\overline Y$.

We prove that the divisor $\Sigma _{t_i\in\curly T(\overline K)}C_i\in(\Pic\overline Y/2\Pic\overline X)^{G_K}$ and that its image of under the isomorphism \eqref{piciso1}, is the cocycle $\alpha\colon G_K\rightarrow \Pic\overline Y$ defined by
\begin{equation}
\label{eq:fixed}
\sigma\mapsto\begin{cases}-H+\Sigma _{t_i\in\curly T(\overline K)}C_i & \textup{ if }\sigma\notin G_{K_{\curly T}}\\ 0 & \textup{ otherwise}.\end{cases}
\end{equation}
By definition, for every $t_i\in\curly S(\overline K)$ the pair of conics $\{C_i,C_i^{'}\}$ is defined over $K(t_i)$ and each individual conic is defined over $K(t_i,\sqrt{\epsilon_{t_i}})$. Hence $\Sigma _{t_i\in\curly T(\overline K)}C_i$ is fixed by $\sigma$ for every $\sigma\in G_{K_{\curly T}}$. Furthermore, if $\sigma\notin G_{K_{\curly T}}$ then 
$$\Sigma _{t_i\in\curly T(\overline K)}C_i-\sigma(\Sigma _{t_i\in\curly T(\overline K)}C_i)=\Sigma _{t_i\in\curly T(\overline K)}C_i-\Sigma _{t_i\in\curly T(\overline K)}C_i^{'}=-2H+2\Sigma _{t_i\in\curly T(\overline K)}C_i.$$
By the previous two statements, $\Sigma _{t_i\in\curly T(\overline K)}C_i\in(\Pic\overline Y/2\Pic\overline Y)^{G_K}$. 
Moreover, by the explicit description of ismorphism \eqref{piciso}, $\alpha\in \Ho^1(K,\Pic \overline Y)$ as defined above is the image of $\Sigma _{t_i\in\curly T(\overline K)}C_i$ by the isomorphism \eqref{piciso1}. 

Therefore $\alpha$ is trivial if and only if $\Sigma _{t_i\in\curly T(\overline K)}C_i\in 2\Pic\overline Y+(\Pic\overline Y)^{G_K}$. We will prove that $\Sigma _{t_i\in\curly T(\overline K)}C_i\notin 2\Pic\overline Y+(\Pic\overline Y)^{G_K}$ if and only if there exists $T\in\curly S-\curly T$ such that $\epsilon_{T}\notin K(T)^{\times 2}$. We determine an equivalent criterion to $\Sigma _{t_i\in\curly T(\overline K)}C_i\notin 2\Pic\overline Y+(\Pic\overline Y)^{G_K}$ by using the generators of $\Pic\overline Y$. By Equation \eqref{eq:fixed}, $\Sigma _{t_i\in\curly T(\overline K)}C_i\notin(\Pic\overline Y)^{G_K}$. Moreover, by the same argument as in the proof of Lemma \ref{lem:H1} any combination of the generators of $\Pic\overline Y$ involving an odd coefficient of $\frac{H+\Sigma_iC_i}{2}$ is not fixed by $G_K$. Therefore $\Sigma _{t_i\in\curly T(\overline K)}C_i\in 2\Pic\overline Y+(\Pic\overline Y)^{G_K}$ if and only if there exists a choice of signs such that 
$$
\Sigma _{t_i\in(\curly S-\curly T)(\overline K)}\pm C_i\in(\Pic \overline Y)^{G_K}.
$$
If there exists $T\in\curly S-\curly T$ such that $\epsilon_{T}\notin K(T)^{\times 2}$ then by Proposition \ref{prop:action} $G_K$ acts transitively on each pair $\{C_i,C_i^{'}:t_i\in T(\overline K)\}$. Then for some $t_j\in T(\overline K)$ and by using the fact that $C_j^{'}=H-C_j$, there exists $\sigma\in G_K$ such that
$$
\sigma(\pm C_j+\Sigma_{\{t_i\neq t_j, t_i\in(\curly S-\curly T)(\overline K)\}}\pm C_i)=\mp C_j+D
$$
where $D$ is a linear combination of the $C_i$'s excluding $C_j$. Then for any choice of signs
$$
\Sigma _{t_i\in(\curly S-\curly T)(\overline K)}\pm C_i\notin(\Pic \overline Y)^{G_K}.
$$
This proves that $\alpha$ is a non trivial cocycle in $\Ho^1(K,\Pic\overline Y)$.

Conversely, if for every $T\in\curly S-\curly T$ we have $\epsilon_T\in K(T)^{\times 2}$ then for any choice of signs
$$
\Sigma _{t_i\in(\curly S-\curly T)(\overline K)}\pm C_i\in(\Pic \overline Y)^{G_K}.
$$
So $\alpha$ is trivial in $\Ho^1(K,\Pic\overline Y)$.
\end{proof}

Let $V$ be a nice scheme over a field $K$. Let $L$ be a cyclic extension of $K$; denote by $\sigma$ the generator of $\Gal(L/K)$. Let $f$ be any element in $K(V)^{\times}$. We denote by $\Br_{cyc}(V,L)$ the set of classes of algebras $[(L/K,f)]$ in the image of $\Br V/\Br_0V\rightarrow \Br K(V)/\Br_0V$ where $\Br_0V=\im(\Br K\rightarrow\Br V)$. We view $1-\sigma$ as endomorphisms of $\textup{Div}V_L$. We let $N_{L/K}\colon\textup{Div}V_L\rightarrow\textup{Div}V_K$ and $\overline N_{L/K}\colon\Pic V_L\rightarrow\Pic V_K$ be the usual norm maps.

\begin{thm}
\label{thm:brcyc}
There exists an injection $\Br_{cyc}(V,L)\hookrightarrow\Ho^1(K,\Pic\overline V)$ given by the composition of the maps
\begin{equation}
\Br_{cyc}(V,L)\xrightarrow{\sim}\frac{\ker(\overline{N}_{L/K})}{\im(1-\sigma)}\hookrightarrow \Ho^1(K,\Pic\overline V).
\end{equation}
The image of the class $[(L/K,f)]$ under the above composition is the cocycle 
$$
\sigma\mapsto\begin{cases}D & \textup{ if }\sigma\notin G_{L}\\ 0 & \textup{ otherwise}\end{cases},
$$
where $D$ is the divisor such that $(f)=N_{L/K}(D)$.\qed
\end{thm}

\begin{proof}
The first isomorphism follows from \cite[Theorem 3.3]{VA}. By \cite[Theorem 3.3]{VA}, the isomorphism maps the class of the algebra $[(L/K,f)]$ to the divisor $D$ such that $$N_{L/k}(D)=(f).$$ 

The extension $L/K$ is cyclic; so by using the explicit resolution we compute 
$$\frac{\ker(\overline N_{L/K})}{\im(1-\sigma)}\simeq\Ho^1(\Gal(L/K),\Pic V_L).$$
The image of $D$ under this isomorphism is the cocycle $\alpha$ that maps $\sigma\mapsto D$.

Furthermore, the first part of the inflation-restriction exact sequence yields the injection
$$
\Ho^1(\Gal(L/K),\Pic V_L)\hookrightarrow \Ho^1(K,\Pic\overline V).
$$
The image of $\alpha$ under the above inflation map is the required cocycle.
\end{proof}

Applying the map in Theorem \ref{thm:brcyc} to the del Pezzo surface $Y$ and the cyclic extension $K_{\curly T}/K$, we get the map
\begin{equation}
\label{eq:map}
\Br_{cyc}(Y,k_{\curly T})\xrightarrow{\sim}\frac{\ker(\overline N_{K_{\curly T}/K})}{\im(1-\sigma)}\hookrightarrow \Ho^1(K,\Pic\overline Y).
\end{equation}

For the remainder of this section, we assume that for every subscheme $\curly T\subset \curly S$ satisfying $(*)$, the quadric $Q_T$ has a smooth $K(T)$-point for every $T\in\curly T$. Let 
$$
\curly A_{\curly T}:=\left(K_{\curly T}/K,l^{-2}\Pi_{T\in\curly T}\textup{N}_{K(T)/K}(l_T)\right),
$$
where $l_T$ is a $K(T)$-linear form such that the associated hyperplane is tangent to $Q_T$ at a smooth point for every $T\in\curly T$ and $l$ is any linear form.
\begin{prop}
\label{prop:algebra}
Let $\curly T\subset \curly S$ satisfy $(*)$. The cyclic algebra $\curly A_{\curly T}$
is in the image of $\Br Y\rightarrow \Br K(Y)$.
Further, if there exists $T\in\curly S-\curly T$ such that $\epsilon_{T}\notin K(T)^{\times 2}$ then the algebra $\curly A_{\curly T}$ is non trivial. In particular, it maps under the map \ref{eq:map} to the nontrivial cocycle 
$$
\sigma\mapsto\begin{cases}-H+\Sigma _{t_i\in\curly T(\overline K)}C_i & \textup{ if }\sigma\notin G_{K_{\curly T}}\\ 0 & \textup{ otherwise}\end{cases}
$$
in $\Ho^1(K,\Pic\overline Y)$.
\end{prop}

\begin{proof}
To prove that $\curly A_{\curly T}\in\im(\Br X\rightarrow\Br K(Y))$, we prove that $\curly A_{\curly T}$ is unramified at every codimension 1 point $x\in Y^{(1)}$, i.e., $\partial_x(\curly A)=0$ by the residue sequence, \cite{Gro1,Gro2,Gro3}. The prime divisors which correspond to a valuation such that the function $l^{-2}\Pi_{T\in\curly T}\textup{N}_{K(T)/K}(l_T)$ has an odd valuation at are $C_T$ and $C_T^{'}$ for every $T\in\curly T$. However, for every $T\in\curly T$, $K_{\curly T}\subset \kappa (C_T\cup C_T^{'})$ because $C_T$ and $C_T^{'}$ are conjugate over $K_{\curly T}$. So by \cite[7.5.1]{CSA}, $\curly A$ is unramified at $C_T$ and $C_T^{'}$ for every $T\in\curly T$. Hence $\curly A\in\im(\Br Y\rightarrow \Br K(Y))$ by the residue sequence, \cite{Gro1,Gro2,Gro3}.

By definition of the maps defining \ref{eq:map}, the class of the algebra $\curly A_{\curly T}$ gets mapped to a cocycle $\alpha$ that maps $G_{K_{\curly T}}$ to the identity and maps any element $\sigma\notin G_{K_{\curly T}}$ to a divisor $D$ such that $\textup{N}_{K_{\curly T}/K}(D)=\textup{div}(l^{-2}\Pi_{T\in\curly T}\textup{N}_{K(T)/K}(l_T))=-2H+\Sigma _{t_i\in\curly T(\overline K)}C_i+\Sigma _{t_i\in\curly T(\overline K)}C_i^{'}$. So $D$ can be chosen as $-H+\Sigma _{t_i\in\curly T(\overline K)}C_i$. The cocycle $\alpha$ is non trivial by Proposition \ref{prop:H1}. So $\curly A_{\curly T}$ is non trivial as well.
\end{proof}

\section{An example of a del Pezzo surface of degree 4 with trivial Brauer Group and non trivial $\Ho^1(k,\Pic\overline X)$}

In this section we prove the main theorem.

\begin{thm}
\label{thm:main}
Let $k=\QQ^{cycl}(a,b,c)$ where $a,b$ and $c$ are independent transcendental elements. Let $X$ be the del Pezzo surface of degree 4 in $\PP^4_k$ defined by the intersection of the following two quadrics
\begin{equation*}
\begin{split}
Q&:ax_0^2+bx_1^2+x_2^2+cx_4^2=0\\
Q^{'}&:bcx_0^2+x_1^2+x_2^2+ax_3^2=0.
\end{split}
\end{equation*}
Then $\Ho^1(k,\Pic \overline X)\simeq\ZZ/2\ZZ$ while $\frac{\Br X}{\Br k}$ is trivial.
\end{thm}
The following corollary follows from the algorithm in \cite[Section 4.4]{BT}.

\begin{cor}
The assumption that $Q_T$ has a smooth $k(T)$-point for every $T\in\curly T$ in \cite[Section 4.4]{BT} cannot be omitted entirely to apply the algorithm for computing $\Br X/\Br k$.\qed
\end{cor}

The proof of Theorem \ref{thm:main} relies on the functoriality with respect to base extension from $k$ to $L=k(\sqrt a)$ of the Hochschild-Serre spectral sequence:
\begin{equation*}
0\rightarrow\Pic X\rightarrow (\Pic \overline X)^{G_k}\rightarrow \Br k\rightarrow \Br_1X\rightarrow \Ho^1(k,\Pic \overline X)\rightarrow \Ho^3(k,\GG _m),
\end{equation*}
where $\Br_1 X=\ker(\Br X\rightarrow \Br \overline X)$. Since $X$ is rational then $\Br_1X=\Br X$. Later in this section, we prove that the algorithm in \cite[Section 4.1]{BT} can be applied over $L$ and that the generating Brauer class of $\Br X_L/\Br L$ is not in the image of $\Br X/\Br  k$.

\subsection{Degeneracy locus of $X$}

Let $A$ and $A^{'}$ be the matrices associated to the quadratic forms of the quadrics $Q$ and $Q^{'}$ respectively. The characteristic polynomial of the pencil of quadrics $\{\lambda Q+\mu Q^{'}:[\lambda:\mu]\in\PP^1\}$ is
\begin{equation}
f(\lambda,\mu)=\det(\lambda A+\mu A^{'})= 32\mu\lambda (\lambda+\mu)(b\lambda+\mu)(a\lambda+bc\mu).
\end{equation}
Since all the irreducible factors of $f(\lambda,\mu)$ are distinct, it is separable in $k[\lambda,\mu]$. Therefore \cite[Proposition 3.26]{Wit}, $X$ is smooth.
Moreover, the degeneracy locus $\curly S$ of this del Pezzo surface consists of the five degree 1 points $$T_0=[1:0],\; T_1=[0:1],\;T_2=[1:-1],\;T_3=[1:-b],\;\textup{and }T_4=[bc:-a]$$ corresponding to the linear factors of $f(\lambda,\mu)$.

For each point in the degeneracy locus $\curly S$, we compute the corresponding quadric and discriminant as explained in Chapter 3. We show the case corresponding to $T_0$ and summmarize the rest in an array below.
The singular locus of the quadric $Q_{T_0}=V(ax_0^2+bx_1^2+x_2^2+cx_4^2)$ is $[0:0:0:1:0]$. So we may choose $H_{T_0}$ to be $V(x_3)$. By a direct computation,
$$
\epsilon_{T_0}=16abc\sim abc \mod{k(T_0)^{\times 2}=k^{\times 2}}.
$$
The quadrics and discriminants corresponding to all the points $\{T_0,\ldots,T_4\}$ are summarized below.
\[\begin{array}{rl}
Q_{T_0}\colon ax_0^2+bx_1^2+x_2^2+cx_4^2=0, & \epsilon_{T_0}\sim abc\\
Q_{T_1}\colon bcx_0^2+x_1^2+x_2^2+ax_3^2=0, & \epsilon_{T_1}\sim abc\\
Q_{T_2}\colon (a-bc)x_0^2+(b-1)x_1^2-ax_3^2+cx_4^2=0, & \epsilon_{T_2}\sim ac(b-1)(a-bc)\\
Q_{T_3}\colon(a-b^2c)x_0^2+(1-b)x_2^2-abx_3^2+cx_4^2=0, & \epsilon_{T_3}\sim abc(1-b)(a-b^2c)\\
Q_{T_4}\colon(b^2c-a)x_1^2+(bc-a)x_2^2-a^2x_3^2+bc^2x_4^2=0, & \epsilon_{T_4}\sim b(b^2c-a)(bc-a)
\end{array}\]
 
\subsection{Computing $\Ho^1(k,\Pic\overline X)$ and $\Ho^1(L,\Pic\overline X)$}
The first part of the following Lemma is used to compute $\Ho^1(k,\Pic\overline X)$ and $\Ho^1(L,\Pic\overline X)$ and the second part will be used later in the proof of Theorem \ref{thm:main} to replace $\Br_0X$ and $\Br_0X_L$ by $\Br k$ and $\Br L$ respectively.
\begin{lem}
\label{lem:H1}
Let $K=k$ or $L$, where $k$ and $L$ are the fields defined before. For $X$ as before we have,
\begin{enumerate}
\item
$\frac{(\Pic\overline X/2\Pic\overline X)^{G_K}}{((\Pic \overline X)^{G_K}/2(\Pic\overline X)^{G_K})}=\langle C_0+C_1\rangle$
\item
$(\Pic\overline X)^{G_K}=\langle H\rangle$.
\end{enumerate}
\end{lem}

\begin{proof}
We start by proving that $(\Pic\overline X/2\Pic\overline X)^{G_K}=\langle H,C_0+C_1\rangle$. 
By Proposition \ref{prop:pic}, classes in $\Pic\overline X/2\Pic\overline X$ can be represented by divisors of the form $D=\beta(\frac{H+\Sigma_iC_i}{2})+\Sigma_i\alpha_iC_i$ where $\beta$ and $\alpha_i$ are either 0 or 1  for all $i\in\{0,\ldots,4\}$. 
Let $\sigma$ be any element in $G_K$. By Proposition \ref{prop:action} and since each $T\in\curly S$ has degree 1, $\sigma$ is the product of an even number of exchanges between $C_i$ and $C_i^{'}=H-C_i$ for $i\in\{0,\ldots 4\}$. Let $I\subset\{0,\ldots,4\}$ be the set of indices of the exchanges that are factors of $\sigma$. Since $\sigma$ is the product of an even number of exchanges, $I$ has even cardinality. The Galois element $\sigma$ is determined by the set of indices $I$; so we denote it by $\sigma_I$. We are interested in characterizing $D\in(\Pic\overline X/2\Pic\overline X)^{G_K}$ or equivalently $D\in\Pic\overline X$ such that $D-\sigma_I D\in 2\Pic\overline X$ for every $\sigma_I\in G_K$. First we compute $\sigma_I D$ for any $\sigma_I\in G_K$. Let $E:=\frac{H+\Sigma_iC_i}{2}$ and $\gamma:=\Sigma_{i\in I}\alpha_i$.
\begin{align*}
\sigma_I D
&=\sigma(\beta(E)+\Sigma_i\alpha_iC_i)\\
&=\beta\left(\frac{H+\Sigma_{i\in I}(H-C_i)+\Sigma_{i\notin I}C_i}{2}\right)+\Sigma_{i\notin I}\alpha_iC_i+\Sigma_{i\in I}\alpha_i(H-C_i)\\
&=\beta\left((1+|I|)E-\frac{2+|I|}{2}\Sigma_{i\in I}C_i-\frac{|I|}{2}\Sigma_{i\notin I}C_i\right)+\Sigma_{i\notin I}\alpha_iC_i+2\gamma E\\&-\Sigma_{i\in I}(\gamma+\alpha_i)C_i-\gamma\Sigma_{i\notin I}C_i\\
&=(\beta+\beta |I|+2\gamma)E-\Sigma_{i\in I}\left(\frac{\beta |I|+2\gamma}{2}+\beta+3\alpha_i\right)C_i-\frac{\beta |I|+2\gamma+2\alpha_i}{2}\Sigma_{i\notin I}C_i.
\end{align*}
We expand and arrange $D-\sigma_I D$ as
\begin{align*}
D-\sigma_I D
&=-(\beta |I|+2\gamma)E+\Sigma_{i\in I}\left(\frac{\beta |I|+2\gamma}{2}+\beta+2\alpha_i\right)C_i+\frac{\beta |I|+2\gamma}{2}\Sigma_{i\notin I}C_i.
\end{align*}
By transitivity of the action of $G_K$, Proposition \ref{prop:action}, we may assume $I$ to be non trivial.
By Considering the coefficients of $D-\sigma_I D$ we get
\begin{align*}
\frac{\beta |I|+2\gamma}{2}\equiv 0\pmod{2}, \quad
\frac{\beta |I|+2\gamma}{2}+\beta+2\alpha_j\equiv 0\pmod 2.
\end{align*} 
Hence $\beta=0$. So $\gamma$ is even. 

Now we consider the possibilities of the even cardinality sets $I\subset\{0,\ldots,4\}$. Since $C_0,C_1$ are defined over $K(\sqrt{\epsilon_{T_0}})=K(\sqrt{\epsilon_{T_1}})$, then either both exchanges between $\{C_0,C_0^{'}\}$ and $\{C_1,C_1^{'}\}$ are factors of $\sigma\in G_K$ or both are not. Hence $\#(I\cap\{0,1\})\neq 1$. By the computation of $D-\sigma_I D$ before when $I=\{0,1\}$, we deduce that $\alpha_0+\alpha_1\equiv 0\pmod 2$. 
Since $I$ has even cardinality and $\#(I\cap\{0,1\})\neq 1$ and by transitivity of the action of $G_K$ on $\{C_2,C_2^{'}\}$, $\{C_3,C_3^{'}\}$, and $\{C_4,C_4^{'}\}$, the nontrivial possibilities of $I-\{0,1\}\cap I$ are $\{2,3\}$, $\{2,4\}$, and $\{3,4\}$. The class of the divisor $D\pmod 2$ is fixed by $\sigma_I$ where $I-\{0,1\}\cap I$ is $\{2,3\}$, $\{2,4\}$, or $\{3,4\}$ if and only if $\alpha_2+\alpha_3$, $\alpha_2+\alpha_4$, or $\alpha_3+\alpha_4$ are even respectively. From the discussion before we deduce
$$
\alpha_0+\alpha_1\equiv 0,\quad
\alpha_2+\alpha_3\equiv 0,\quad
\alpha_2+\alpha_4\equiv 0,\quad\textup{and }
\alpha_3+\alpha_4\equiv 0\pmod 2.
$$
Since $\alpha_i\in\{0,1\}$, from the above congruences we deduce that $D$ is $0$, $C_0+C_1$, $C_2+C_3+C_4$, or $\Sigma_iC_i$. Hence $(\Pic\overline X/2\Pic\overline X)^{G_K}=\langle C_0+C_1,C_2+C_3+C_4\rangle$. Further,
$$
C_2+C_3+C_4=C_0+C_1+2(\frac{1}{2}(H+\Sigma C_i))-2C_0-2C_1-H.
$$
So we may rewrite the generators as $(\Pic\overline X/2\Pic\overline X)^{G_K}=\langle H,C_0+C_1\rangle$.

We prove that $\frac{(\Pic\overline X/2\Pic\overline X)^{G_K}}{((\Pic \overline X)^{G_K}/2(\Pic\overline X)^{G_K})}=\langle C_0+C_1\rangle$ or equivalently that $(\Pic\overline X)^{G_K}=\langle H\rangle$. Let $\Sigma_{i}a_iC_i$ be a nontrivial combination of the $C_i$'s. Let $j\in\{0,\ldots 4\}$ be an arbitrary element such that $a_j\neq 0$. By Proposition \ref{prop:action} and the fact that $\epsilon_j$ is not a square in $K(T_j)^{\times}=K^{\times}$, there exists a $\tau\in G_K$ such that $\tau C_j=C_j^{'}=H-C_j$. Therefore, $\tau(\Sigma_{i}a_iC_i)=-\alpha_jC_j+D^{'}$ where $D^{'}$ is a linear combination of $C_i$ $i\neq j$. Hence $\Sigma_{i}a_iC_i$ is not fixed by $G_K$. Therefore $(\Pic\overline X)^{G_K}=\langle H\rangle$ and this proves (1) and (2). 
\end{proof}

\begin{prop}
\label{prop:H1}
For $X$ and $k$, and $L$ as before, we have $\Ho^1(k,\Pic\overline X)\simeq\Ho^1(L,\Pic\overline X)\simeq\ZZ/2\ZZ$ and $\textup{Res}\colon\Ho^1(k\Pic\overline X)\rightarrow \Ho^1(L,\Pic\overline X)$ is an isomorphism.
\end{prop}

\begin{proof}
Consider the following short exact sequence
$$
0\rightarrow\Pic\overline X\xrightarrow{\times 2}\Pic \overline X\rightarrow (\Pic\overline X/2\Pic\overline X)\rightarrow 0.
$$
The connecting morphism in the induced long exact sequence yields
\begin{equation}
\label{piciso}
\frac{(\Pic\overline X/2\Pic\overline X)^{G_K}}{((\Pic \overline X)^{G_K}/2(\Pic\overline X)^{G_K})}\xrightarrow{\sim}\Ho^1(K,\Pic\overline X)[2];\quad [D]\mapsto (\sigma\mapsto\frac{1}{2}(d-\sigma d))
\end{equation}
where $K=k$ or $L$, and $d$ is a lift of $D$ to $\Pic\overline X$. 
By Lemma \ref{lem:H1} (1) and by the isomorphism \eqref{piciso}, we deduce that $\Ho^1(K,\Pic\overline X)[2]\simeq \ZZ/2\ZZ$. Further, $\Ho^1(K,\Pic\overline X)$ is 2-torsion because the cardinality of $\Ho^1(k,\Pic\overline X)$ divides $4$ as verified in \texttt{Magma}~\cite{Magma} script in the \texttt{arXiv} distribution of \cite{BT}. So $\Ho^1(k,\Pic\overline X)\simeq\Ho^1(L,\Pic\overline X)\simeq\ZZ/2\ZZ$. Moreover, the image of the cocycle corresponding to $C_0+C_1$, that generates $\frac{(\Pic\overline X/2\Pic\overline X)^{G_K}}{((\Pic \overline X)^{G_K}/2(\Pic\overline X)^{G_K})}$ by Lemma \ref{lem:H1}, under the restriction map $\textup{Res}\colon\Ho^1(k\Pic\overline X)\rightarrow \Ho^1(L,\Pic\overline X)$ is the cocycle corresponding to $C_0+C_1$. So $\textup{Res}\colon\Ho^1(k,\Pic\overline X)\rightarrow \Ho^1(L,\Pic\overline X)$ is an isomorphism.
\end{proof}

\subsection{The Brauer group of $X$ over $L=k(\sqrt a)$}

We use the algorithm in \cite[Section 4.1]{BT} to prove that $\Br X_L/\Br L\simeq \ZZ/2\ZZ$ and to explicitly construct a non trivial algebra $\curly A$ whose class generates $\Br X_L/\Br L$. This algebra will be used in the proof of Theorem \ref{thm:main}.



\begin{prop}
\label{prop:L}
The Brauer group $\Br X_L/\Br L$ is isomorphic to $\ZZ/2\ZZ$ and is generated by the class of the quaternion algebra:
$$
\curly A=\left(\epsilon_{T_0},\frac{(2aix_0+2\sqrt ax_2)(2ix_1+2x_2)}{(x_0+x_1)^2}\right).
$$
Moreover, $\curly A-\sigma\curly A\sim (c,b)\in\Br L$.
\end{prop}

\begin{proof}
The same computation of the degeneracy locus, associated quadrics, and discriminants of $X$ in Section 4.1 still work over $L$. Moreover $\curly T=\{T_0,T_1\}$ is a degree two subscheme of $\curly S$ such that $\Pi_{T\in\curly T}N_{L(T)/L}=\epsilon_{T_0}\epsilon_{T_1}=a^2b^2c^2\in L(T)^{\times 2}=L^{\times 2}$, and the element $\epsilon_{T_0}=\epsilon_{T_1}=bc$ is non trivial in $L^{\times}/L^{\times 2}$. Hence $\curly T$ satisfies $(*)$ as defined in Section 2. By direct computations we show $\{T_0,T_1\}$ is the only subscheme of $\curly S$ that satisfies $(*)$. Moreover, the quadrics $Q_{T_0}:ax_0^2+bx_1^2+x_2^2+cx_4^2=0$ and $Q_{T_1}: bcx_0^2+x_1^2+x_2^2+ax_3^2=0$ have smooth $L$-points $P_{T_0}=[i:0:\sqrt a:0:0]$ and $P_{T_1}=[0:i:1:0:0]$ respectively. 
Further, $\epsilon_{T_2}$ is not a square in $L^{\times}$. So by the algorithm in \cite[Section 4.1]{BT}, 
$\Br X_L/\Br L\simeq \Ho^1(L,\Pic\overline X)\simeq \ZZ/2\ZZ$ and is generated by the algebra 
$$
\curly A=\left(\epsilon_{T_0},\frac{l_{T_0}l_{T_1}}{l^2}\right)
$$
where $l_T$ is the $L(T)$-linear form such that the associated hyperplane is tangent to $Q_T$ at a smooth point $P_T$ of $Q_T$ for $T\in\curly T$, and $l$ is any linear form. 

Computing the linear forms we get, $l_{T_0}:2aix_0+2\sqrt ax_2$, and $l_{T_1}:2ix_1+2x_2$. Substituting these into $\curly A$ yields the required algebra.

We have
\begin{align*}
&\curly A-\sigma \curly A\\&=\left(bc,\frac{(2aix_0+2\sqrt ax_2)(2ix_1+2x_2)}{(x_0+x_1)^2}\right)\left(bc,\frac{(2aix_0-2\sqrt ax_2)(2ix_1+2x_2)}{(x_0+x_1)^2}\right)\\&=\left(bc, \frac{-4(\sqrt a)^2(ax_0^2+x_2^2)(2ix_1+2x_2)^2}{(x_0+x_1)^4}\right)\\
&\sim \left(bc, \frac{-(ax_0^2+x_2^2)}{(x_0+x_1)^2}\right)
\end{align*}
By the defining equation of $Q^{'}$, $ax_0^2+x_2^2=-bx_1^2-cx_4^2=-b(x_1^2-cb(i\frac{x_4}{b})^2)$.
So $\curly A-\sigma\curly A\sim (bc,b)\sim (c,b)$.
\end{proof}

\subsection{Characterizing $\im (\Br X/\Br k\rightarrow \Br X_L/\Br L)$}
Let $\Gal(L/k)=\langle\sigma\rangle$.
\begin{lem}
\label{lem:fixed}
Let $\Gal(L/k)=\langle\sigma\rangle$.
If $\curly A-\sigma\curly A\neq x-\sigma x$ for every $x\in\Br L$ then $[\curly A]\notin \im (\Br X/\Br k\rightarrow \Br X_L/\Br L)$.
\end{lem}

\begin{proof}
By \cite[Proposition 3.3.17]{CSA}, the generalized inflation-restriction sequence for the field extension $L(X)/k(X)$ is
$$
0\rightarrow \Ho^2(\Gal(L/k),L(X)^{\times})\xrightarrow{\textup{inf}} \Ho^2(k(X),\overline {k(X)}^{\times})\xrightarrow{\textup{Res}} \Ho^2(L(X),\overline {k(X)}^{\times})^{\Gal(L/k)}.
$$
If the class of $\curly A$ is in the image of $\Br X/\Br k\rightarrow \Br X_L/\Br L$, then there exists $x\in\Br L$ such that $\curly A-x\in \im (\Br X\rightarrow \Br X_L)$. Hence $\curly A-x\in \im (\Br k(X)\rightarrow \Br L(X))$. By the generalized inflation-restriction sequence above, $\curly A -x$ is fixed by $\Gal(L/k)=\langle\sigma\rangle$. Rearranging we get that $\curly A-\sigma\curly A=x-\sigma x$.
\end{proof}



\subsection{Proof of Theorem \ref{thm:main}}


First we may replace $\Br_0 X$ and $\Br_0 X_L$ by $\Br k$ and $\Br L$ respectively because the map $(\Pic\overline X)^{G_K}=\langle H\rangle\rightarrow \Br K$ is trivial by the exact sequence that follows from the Hochschild-Serre spectral sequence where $K=k$ or $L$.

For the sake of contradiction we assume that $\Br X/\Br k$ is non trivial. Since there is an injection $\Br X/\Br k\hookrightarrow \Ho^1(k,\Pic\overline X)\simeq \ZZ/2\ZZ$ by the exact sequence that follows from the Hochschild-Serre spectral sequence and $\Br X/\Br k$ is nontrivial, there is a unique nontrivial class in $\Br X/\Br k$; denote this class by $[\curly B]$. By Proposition \ref{prop:H1}, Proposition \ref{prop:L} and the functoriality of the Hochschild-Serre spectral sequence we have
$$
\begin{tikzcd}
\Br X/\Br k\dar\arrow[r,"\simeq"] & \Ho^1(k,\Pic\overline X)\arrow[d,"\simeq"]\\
\Br X_L/\Br L\arrow[r,"\simeq"] & \Ho^1(L,\Pic\overline X).
\end{tikzcd}
$$
So $[\curly B]\in \Br X/\Br k$ gets mapped to $[\curly A]\in \Br X_L/\Br L$ as defined in Proposition \ref{prop:L} by the field extension map. We will show that any algebra in the class of $[\curly A]\in \Br X_L/\Br L$ is not in the image of the map $\Br X\rightarrow \Br X_L$, thus resulting in a contradiction. By Lemma \ref{lem:fixed} it is enough to prove that $\curly A-x$ is not fixed by $\sigma$ for all $x\in\Br L$.

Suppose that there exists $x\in\Br L$ such that $\curly A-x$ is fixed by $\sigma$, i.e., $\curly A-\sigma\curly A=x-\sigma x$. By Proposition \ref{prop:L}, both sides of the equation $\curly A-\sigma A=x-\sigma x$ are in $\Br L$. Let $Y=\spec{\QQ^{cycl}(b,c)[\sqrt a]}$. Let $P:\sqrt a=0$ be a divisor on $Y$. Because $L$ is the function field of $Y$ and $P$ is a divisor on $Y$, we have the residue sequence, (\cite{Gro1}, \cite{Gro2},\cite{Gro3}),
$$
\Br\curly O_{Y,\eta}\rightarrow \Br L\xrightarrow{\partial_P}\Ho^1(\QQ^{cycl}(b,c)(P),\QQ/\ZZ).
$$
If $x$ is cyclic of degree $n$ then $\partial_P(x)$ is determined by a degree $n$ cyclic extension of $\QQ^{cycl}(b,c)$ and a choice of a generator of the Galois group of the extension. By Kummer Theory the cyclic extension determined by $\partial_P(x)$ is of the form $\QQ^{cycl}(b,c)(\sqrt[n]{\alpha})/\QQ^{cycl}(b,c)$ where $\alpha\in\QQ^{cycl}(b,c)^{\times}$.
Let $f=f(b,c)\in L$ be a function of $b$, and $c$. By \cite[7.5.1]{CSA}, the residue $\partial_P$ of the cyclic algebra $(\sqrt a,f(b,c))_n\in\Br L$ is determined by the cyclic extension $$\QQ^{cycl}(b,c)(\sqrt[n]{(\sqrt a) ^{-v_P(f)}f^{v_P(\sqrt a)}})/\QQ^{cycl}(b,c)$$ and a choice of generator of the Galois group of this extension. Therefore we may choose $f$ such that $\partial_P((\sqrt a,f(b,c))_n)=\partial_P(x)$. Furthermore $\sigma((\sqrt a, f(b,c))_n)=(-\sqrt a,f(b,c))_n=(\sqrt a, f(b,c))_n$ because $-1$ is an $n$-th root of unity in $L$. Since $(\sqrt a,f(a,b))_n$ is fixed by $\sigma$, subtracting $(\sqrt a, f(b,c))$ from $x$ will not change $x-\sigma x$. Moreover, $(\sqrt a, f(b,c))$ has the same residue as $x$ at $P$; so we may assume without loss of generality that $x$ is unramified at $P$. By Murkurjev-Suslin Theorem and the fact that $\partial_P$ is a homomorphism we extend this argument to any central simple algebra $x\in\Br L$. So we assume $x$ is unrafimied at $P$ in general. We claim that $\curly A-\sigma\curly A$ is also unramified at $P$ and we prove it later. Hence we may specialize the equation $\curly A-\sigma\curly A=x-\sigma x$ at the divisor $P$,
$$
(\curly A-\sigma\curly A)_{|P}=x_{|P}-(\sigma x)_{|P}.
$$
The action of $\sigma$ on $x$ commutes with specialization. Furthermore, $P$ is invariant under $\sigma$. Hence $(\sigma x)_{|P}=\sigma x_{|P}=x_{|P}$. So $(\curly A-\sigma\curly A)_{|P}$ is trivial in $\Br \kappa$ where $\kappa$ is the residue field of $Y$ at $P$. In our example $\kappa = \QQ^{cycl}(b,c)$. Since $\curly A-\sigma\curly A\sim (b,c)$, by Proposition \ref{prop:L}, is a constant Azumaya algebra in $\Br\curly O_{Y,\eta}$, $(\curly A-\sigma\curly A)_{|P}=\curly A-\sigma\curly A\sim (b,c)$. Hence $(b,c)$ is trivial in $\Br \QQ^{cycl}(b,c)$. This is a contradiction because the extension $\QQ^{cycl}(b)(\sqrt{c^{-v_D(b)}b^{v_D(c)}})/\QQ^{cycl}(b)$ associated to $\partial _{c=0}((c,b))$ by \cite[7.5.1]{CSA} is non trivial.\qed

Now we prove the claim that the algebra $\curly A-\sigma\curly A$ is unramified at $P\colon \sqrt a=0$.
We have the residue sequence, (\cite{Gro1}, \cite{Gro2},\cite{Gro3})
$$
\Br\curly O_{Y,\eta}\rightarrow \Br L\xrightarrow{\partial_P}\Ho^1(\QQ^{cycl}(b,c)(P),\QQ/\ZZ).
$$
By \cite[7.5.1]{CSA}, the residue $\partial_P$ of the cyclic algebra $\curly A-\sigma\curly A\sim (b,c)\in\Br L$ is determined by the cyclic extension $\QQ^{cycl}(b,c)(\sqrt{ b ^{-v_P(c)}c^{v_P(b)}})/\QQ^{cycl}(b,c)$ and a choice of a generator of the Galois group of this extension. The valuations $v_P(c)=v_P(b)=0$. By the last two sentences, we deduce that $\partial _P(\curly A-\sigma\curly A)=0$. Therefore $\curly A-\sigma\curly A$ is unramified at $P$.\qed

\newpage
\bibliography{refs-1}{}

\begin{thebibliography}{Man74}

\bibitem[BBFL]{BBFL}
M.~J. Bright, N.~Bruin, E.~V. Flynn, and A.~Logan.
\newblock The brauer-manin obstruction and sh[2].
\newblock 10:354--377 (electronic).

\bibitem[BCP]{Magma}
Wieb Bosma, John Cannon, and Catherine Playoust.
\newblock The magma algebra system. i. the user language.
\newblock 24(3-4):235--265.
\newblock Computational algebra and number theory (London, 1993).

\bibitem[Cor]{Corn}
Patrick Corn.
\newblock The brauer-manin obstruction on del pezzo surfaces of degree 2.
\newblock 95(3):735--777.

\bibitem[CTKS]{CTcubic}
Jean-Louis Colliot-Thélène, Dimitri Kanevsky, and Jean-Jacques Sansuc.
\newblock Arithmétique des surfaces cubiques diagonales.
\newblock In {\em Diophantine approximation and transcendence theory (Bonn,
  1985)}, volume 1290 of {\em Lecture Notes in Math.}, pages 1--108. Springer,
  Berlin.

\bibitem[EJa]{EJ}
Andreas-Stephan Elsenhans and Jörg Jahnel.
\newblock On the brauer-manin obstruction for cubic surfaces.
\newblock 2(2):107--128.

\bibitem[EJb]{EJ12}
Andreas-Stephan Elsenhans and Jörg Jahnel.
\newblock On the order three brauer classes for cubic surfaces.
\newblock 10(3):903--926.

\bibitem[Groa]{Gro1}
Alexander Grothendieck.
\newblock Le groupe de brauer. i. algèbres d'azumaya et interprétations
  diverses [ {MR}0244269 (39 \#5586a)].
\newblock In {\em Séminaire Bourbaki, Vol. 9}, pages Exp. No. 290,199--219.
  Soc. Math. France, Paris.

\bibitem[Grob]{Gro2}
Alexander Grothendieck.
\newblock Le groupe de brauer. {II}. théorie cohomologique.
\newblock In {\em Dix exposés sur la cohomologie des schémas}, volume~3 of
  {\em Adv. Stud. Pure Math.}, pages 67--87. North-Holland, Amsterdam.

\bibitem[Groc]{Gro3}
Alexander Grothendieck.
\newblock Le groupe de brauer. {III}. exemples et compléments.
\newblock In {\em Dix exposés sur la cohomologie des schémas}, volume~3 of
  {\em Adv. Stud. Pure Math.}, pages 88--188. North-Holland, Amsterdam.

\bibitem[GS]{CSA}
Philippe Gille and Tam{á}s Szamuely.
\newblock {\em Central simple algebras and Galois cohomology}, volume 101 of
  {\em Cambridge Studies in Advanced Mathematics}.
\newblock Cambridge University Press, Cambridge.

\bibitem[Hara]{Harpaz}
Yonatan. Harpaz.
\newblock The brauer group of the generic affine quadric.

\bibitem[Harb]{Hart}
Robin Hartshorne.
\newblock {\em Algebraic geometry}.
\newblock Springer-Verlag, New York-Heidelberg.
\newblock Graduate Texts in Mathematics, No. 52.

\bibitem[KST]{KST}
B.~EH. Kunyavskij, A.~N. Skorobogatov, and M.~A. Tsfasman.
\newblock del pezzo surfaces of degree four.
\newblock (37):113.

\bibitem[Man74]{CubicForms}
Yu.~I. Manin.
\newblock {\em Cubic forms: algebra, geometry, arithmetic}.
\newblock North-Holland Publishing Co., Amsterdam-London; American Elsevier
  Publishing Co., New York, 1974.
\newblock Translated from the Russian by M. Hazewinkel, North-Holland
  Mathematical Library, Vol. 4.

\bibitem[Uema]{quadric}
Tetsuya Uematsu.
\newblock On the brauer group of affine diagonal quadrics.
\newblock 163:146--158.

\bibitem[Uemb]{cubic}
Tetsuya Uematsu.
\newblock On the brauer group of diagonal cubic surfaces.
\newblock 65(2):677--701.

\bibitem[VAa]{tony}
Anthony V{á}rilly-Alvarado.
\newblock Arithmetic of del pezzo surfaces.
\newblock In {\em Birational geometry, rational curves, and arithmetic}, pages
  293--319. Springer, New York.

\bibitem[VAb]{VA}
Anthony Várilly-Alvarado.
\newblock Weak approximation on del pezzo surfaces of degree 1.
\newblock 219(6):2123--2145.

\bibitem[VAV]{BT}
Anthony V{á}rilly-Alvarado and Bianca Viray.
\newblock Arithmetic of del pezzo surfaces of degree 4 and vertical brauer
  groups.
\newblock 255:153--181.

\bibitem[Wit]{Wit}
Olivier Wittenberg.
\newblock {\em Intersections de deux quadriques et pinceaux de courbes de genre
  1/Intersections of two quadrics and pencils of curves of genus 1}, volume
  1901 of {\em Lecture Notes in Mathematics}.
\newblock Springer, Berlin.

\end{thebibliography}
\bibliographystyle{alpha}
\vspace{4ex}
\noindent\scshape{E-mail address:}
\normalfont\bf\href{mailto:mnr6@uw.edu}{\lowercase{mnr6@math.washington.edu}}
\end{document}